\documentclass[11pt]{amsart}

\usepackage{graphicx,caption,subcaption}

\newcommand{\noop}[1]{} 

\usepackage{amssymb,amsmath,amsfonts,amsthm,amscd,amstext,mathtools}
\usepackage{epstopdf}
\usepackage{microtype}
\usepackage{mathrsfs}
\usepackage{enumitem}	
\usepackage{verbatim} 
\usepackage{url} 
\usepackage{hyperref}
\usepackage{xcolor}
\usepackage{tikz}

\usepackage{cancel}

\newcommand{\ind}{{\sf 1}}

\newcommand{\bP}{\mathbf{P}}
\newcommand{\bQ}{\mathbf{Q}}

\newcommand{\bE}{\mathbf{E}}

\newcommand{\bbR}{\mathbb{R}}
\newcommand{\R}{\mathbb{R}}

\newcommand{\N}{\mathbb{N}}
\newcommand{\bbZ}{\mathbb{Z}}

\newcommand{\cE}{{\ensuremath{\mathcal E}} }

\newcommand{\cC}{{\ensuremath{\mathcal C}} }

\newcommand{\cD}{{\ensuremath{\mathcal D}} }

\renewcommand{\epsilon}{\varepsilon}
\renewcommand{\phi}{\varphi}

\newcommand{\gb}{\beta}
\newcommand{\gd}{\delta}
  
\newcommand{\eps}{\varepsilon}      

\newcommand{\go}{\omega}

\newcommand{\gs}{\sigma}
\newcommand{\gS}{\Sigma}

\newcounter{cst}[section]		
\newcounter{svf}[section]		

\newtheorem{theorem}{Theorem}[section]
\newtheorem{proposition}[theorem]{Proposition}
\newtheorem{corollary}[theorem]{Corollary}
\newtheorem{lemma}[theorem]{Lemma}

\theoremstyle{definition}

\newtheorem{definition}[theorem]{Definition}
\newtheorem*{definition*}{Definition}

\theoremstyle{remark}
\newtheorem{remark}{Remark}[section]

\newtheorem*{notation*}{Notation}
\newtheorem{example}[remark]{Example}

\numberwithin{equation}{section}			

\newcommand{\dd}{\mathrm{d}}		

\renewcommand{\preceq}{\preccurlyeq}		
\renewcommand{\succeq}{\succcurlyeq}		

\renewcommand{\hat}{\widehat}
\renewcommand{\tilde}{\widetilde}

\renewcommand{\hat}{\widehat}
\renewcommand{\tilde}{\widetilde}

\newcommand{\Bor}{\mathrm{Bor}}
\newcommand{\cX}{\mathcal{X}}

\newcommand{\bbp}{\boldsymbol{p}}

\newcommand{\ceq}{\coloneqq}
\newcommand{\ceqq}{\coloneqq}
\newcommand{\Supp}{\mathrm{Supp}}

\title[Some FKG inequalities for stochastic processes]{Some FKG inequalities for stochastic processes}

\author[A. Legrand]{Alexandre Legrand}
\address{Université Claude Bernard Lyon 1, Institut Camille Jordan, UMR 5208. 43 boulevard du 11 novembre 1918, 69622 Villeurbanne Cedex, France}
\email{legrand@math.univ-lyon1.fr}


\makeatletter
\@namedef{subjclassname@2020}{%
	2020 Mathematics Subject Classification}
\makeatother

\date{}

\begin{document}

\begin{abstract}
This paper is interested in proving correlation inequalities of the FKG-type for various stochastic processes in continuous time. The pivotal tool which yields these correlation inequalities is an approximation with (possibly conditioned) Markov chains and random walks. In particular, we prove FKG inequalities for L\'evy processes, Bessel processes and several conditioned Brownian processes. As a side result, we also provide a necessary and sufficient condition for a random walk distribution in $\bbZ$ to satisfy the well-known ``FKG lattice condition''.
\end{abstract}

\maketitle

The \emph{FKG inequality} is a correlation inequality named after Fortuine, Kasteleyn and Ginibre~\cite{FKG71}, who proved a sufficient condition for its validity in the Ising and other lattice models (see also~\cite{Har60}). It has seen many applications and generalizations in statistical mechanics and general probability theory, see e.g.~\cite{BFKL97, Caf00, Hol74, KKO77, Lig05book, Mez97, Pres74} and references therein, or more recently~\cite{Lam22, Lam23-Dic, LO24, She05}; and it is a recurrent tool in proofs of stochastic domination results, as well as computations of lower bounds on first moments or event probabilities. 
It is also equivalent to the notion of \emph{``association''} of random variables, which has been extensively studied, see e.g.~\cite{EPW67}. In~\cite{Bar05}, Barbato proved that the FKG inequality holds for the Brownian motion, more precisely on the space $\cC([0,T],\R)$ of continuous functions from $[0,T]$ to $\R$ endowed with the canonical Wiener measure; and this statement in a functional space has already seen several applications: the reader can refer to~\cite{BM19, GH23, MM18, SSS14} for some examples. 

The present paper is interested in FKG inequalities on functional spaces, and aims to contribute to this literature in a few directions. First, we present a necessary and sufficient condition for a $d$-dimensional L\'evy process to satisfy the FKG inequality: not only this extends Barbato's result beyond the 1-dimensional Brownian case, but this also provides an alternate proof. Then, we show an FKG inequality for a class of conditional distributions of discrete-time Markov chain trajectories. Not only this allows us to prove that many other stochastic processes in continuous time satisfy the FKG inequality, such as Bessel processes and conditioned Brownian processes; this result also relates to several important issues around FKG inequalities, such as establishing necessary and sufficient conditions for the well-known \emph{FKG lattice condition} to hold, or obtaining FKG inequalities for conditional distributions ---a topic which has notably aroused interest in the domains of spin systems (see~\cite{Lig05} and references therein) as well as percolation~\cite{vdBHK06,vdBK01}.

\section{Main results}
\subsection{The FKG inequality and L\'evy processes}
Throughout this paper, $(E,\cE,\preceq)$ denotes a measurable set endowed with a measurable (partial) order, i.e.\ it satisfies $\{(u,v)\in E^2\,;u\preceq v\}\in\cE^{\otimes 2}$. A function $f:E\to\R$ is said to be \emph{increasing} (or non-decreasing) on $E$ if, for all $u,v\in E$, $u\preceq v$ implies $f(u)\leq f(v)$. This paper will mostly be interested in the following cases: 

$(i)$ $E=\R^d$, $\cE=\Bor(\R^d)$ and $\preceq$ is the \emph{product partial order}: i.e. for $u,v\in E$, $u\preceq v$ if and only if $u_i\leq v_i$ for all $1\leq i\leq d$,

$(ii)$ $E=\cC([0,T],\R)$, $\cE=\cE_\cC$ is the $\sigma$-algebra induced by the uniform convergence topology, and for $u,v\in E$, $u\preceq v$ if and only if $u(t)\leq v(t)$ for all $t\in[0,T]$.

$(iii)$ $E=\cD([0,T],\R^d)$ is the set of right-continuous with left-hand limits (rcll) functions on $[0,T]$ taking values in $\R^d$, $\cE=\cE_\cD$ is the $\sigma$-algebra induced by the Skorokhod metric, and for $u=(u_1,\ldots,u_d)$ and $v=(v_1,\ldots,v_d)$ in $\cD([0,T],\R^d)$, $u\preceq v$ if and only if $u_i(t)\leq v_i(t)$ for all $t\in[0,T]$, $1\leq i\leq d$.

\begin{definition}\label{def:fkg}
Let $\bP$ be a probability measure on $(E,\cE,\preceq)$, and $X$ be a random variable with law $\bP$: then $X$ is \emph{positively associated} (or \emph{positively correlated}) if for all increasing functions $f,g\in L^2(\bP)$, one has
\begin{equation}\label{eq:fkg}
\bE[f(X)g(X)] \;\geq\; \bE[f(X)]\,\bE[g(X)]\,. 
\end{equation}
\end{definition}
We shall abusively write that $\bP$ is positively associated whenever $X\sim\bP$ is. When $X$ is positively associated, we may also say that it ``satisfies the FKG inequality''. Our first result is a necessary and sufficient condition for a $d$-dimensional L\'evy process to satisfy the FKG inequality (in particular it encompasses~\cite[Theorem~4]{Bar05}).
\begin{theorem}\label{thm:fkg:levy}
Let $X=(X_t)_{t\in[0,T]}$ be a $d$-dimensional L\'evy process, $d\geq1$, started from $X_0=0$. Let $\nu$ denote its L\'evy measure on $\R^d$, and $\gS=(\gs_{i,j})_{1\leq i,j \leq d}$ the covariance matrix of its Gaussian component. Then, $X$ is positively associated in $\cD([0,T],\R^d)$ if and only if the following two conditions hold:

$(i)$ $\gs_{i,j}\geq0$ for all $1\leq i,j\leq d$,

$(ii)$ $\Supp(\nu) \subset (\R_+)^d\cup(\R_-)^d$. 

\noindent In particular, if $X$ is a 1-dimensional L\'evy process then it is positively associated.
\end{theorem}

The matter of association for infinitely divisible vectors has already been studied before, see notably~\cite{Pit82, Res88, Sam95} and the notion of \emph{``infinite association''} in~\cite{HPAS98}. However, to the author's knowledge, the association of L\'evy processes on the functional space $\cD([0,T],\R^d)$ has not been considered previously. 
The proof of this theorem relies on the notion of infinite association, together with classical arguments on random walks and weak convergence of measures. Let us mention that, even though the proof displayed here is simpler than that of~\cite{Bar05} (which relies on non-trivial measure theory and the introduction of a new order relation on $\cC([0,T],\R^d)$) in the case of the Brownian motion, it seems that some of the arguments in~\cite{Bar05} could be extended to continuous random processes with non-stationary increments.

Let us also mention that, throughout this paper, we will only consider processes which are started from a deterministic initial value, e.g. $X_0=0\in\R^d$ almost surely; but all our results can be extended to generic (random) initial values through the following proposition. For a general random variable $X$ on some space $(E,\cE)$, we denote its law with $\bP_X$.

\begin{definition}\label{def:pogroup}
Let $E$ be a group and $\preceq$ be a partial order on $E$. We say that $(E,\preceq)$ is a partially ordered topological group (or \emph{pogroup}) if it satisfies the following:
\begin{itemize}
\item $E$ is an abelian topological group,
\item for all $x,y,z\in E$, $x\preceq y \Rightarrow x+z\preceq y+z$,
\item the non-negative cone $\cC_+\coloneqq\{x\in E, 0\preceq x\}$ is closed.
\end{itemize}
\end{definition}

\begin{proposition}\label{prop:initialvalue}
Let $(E,\Bor(E),\preceq)$ 
be a pogroup endowed with its Borel $\gs$-algebra $\Bor(E)$. 
Let $X,Y$ be random variables on $E$, 
and assume the following:
\begin{enumerate}
\item $\bP_{X}$ is positively associated,
\item for almost every $X$, $\bP_{Y}(\cdot|X)$ is positively associated,
\item for any $f:E\to\R$ bounded and increasing, $x\mapsto \bE[f(X+Y)|X=x]$ is increasing $\bP_{X}$-almost everywhere,
\end{enumerate}
\noindent then, $\bP_{X+Y}$ is positively associated.
\end{proposition}
Notice that assumptions~2 and~3 both hold if $X,Y$ are independent and $\bP_Y$ is positively associated. Moreover, if $E$ is either $\cC([0,T],\R^d)$ or $\cD([0,T],\R^d)$ and $X$ is an a.s.\ constant process, i.e. $X_t=X_0$ for all $t\in[0,T]$, then Proposition~\ref{prop:initialvalue} amounts to extending the FKG inequality from a process $Y$ to a copy started from a different initial value.

\subsection{Conditional association of Markov chains}
The second result of this paper is that many Markov chain (or random walk) distributions, after conditioning on a quite generic event, are positively associated. 
In particular, we deduce from this the association of several stochastic processes in continuous time which can be approximated by well-chosen Markov chains. Let us introduce the following definitions. 

\begin{definition}\label{def:maxminstable}
We say that an event $A\in\cE$ is \emph{max/min-stable} if for $u,v\in A$, one has $u\vee v\in A$ and $u\wedge v\in A$, where $u\vee v$ (resp. $u\wedge v$) denotes the smallest common upper bound (resp. largest lower bound) of $u$ and $v$ in $E$.
\end{definition}
In this paper we always assume that for all $u,v\in E$, $u\vee v$ and $u\wedge v$ are well-defined in $E$ ---which is the case if $E=\R^d$, $E=\cC([0,T],\R)$ or $E=\cD([0,T],\R^d)$.

\begin{definition}\label{def:crossings}
Let $\bbp\ceqq (p_{x,y})_{x,y\in\cX}$ be a probability transition kernel on some countable subset $\cX\subset\bbR$ (i.e. $p_{x,y}\geq0$ and for all $x\in\cX$, $\sum_yp_{x,y}=1$). For $\cX_1\subset\cX$, we say that $\bbp$ has \emph{$\cX_1$-unfavorable crossings} (abbreviated $\cX_1$-u.c.) if for all $u_1,v_1\in\cX_1$, $u_2,v_2\in\cX$ one has,
\begin{equation}\label{eq:def:crossings}
p_{u_1, u_2}\times p_{v_1, v_2} \,\leq\, p_{u_1\vee v_1, u_2\vee v_2}\times p_{u_1\wedge v_1, u_2\wedge v_2}\,.
\end{equation}
For $\bP$ a (purely atomic) probability measure on $\R$, we say that $\bP$ has $\cX_1$-u.c.\ if the transition kernel of the random walk with increment distribution $\bP$ has $\cX_1$-u.c..
\end{definition}

For the sake of clarity, we now provide some examples of Markov chains on $\cX\ceqq\bbZ$ which have unfavorable crossings. If $\bP$ is an atomic measure on some measurable space $E$, we write abusively $\bP(u)\ceqq\bP(\{u\})$ for $u\in E$.
\begin{lemma}\label{lem:examples}
$(i)$ \emph{Discrete Laplace random walk:} let $\beta>0$ and $\bP(z)\ceqq Ce^{-\beta |z|}\ind_{z\in\bbZ}$, where $C=C(\beta)$ is a normalizing constant. Then $\bP$ has $\bbZ$-u.c. (i.e. the transition kernel of the random walk with increment distribution $\bP$ has $\bbZ$-u.c.). 

$(ii)$ \emph{Markov chain with unit steps:} assume $p_{i,i+1}+p_{i,i-1}=1$ for all $i\in\bbZ$. Then $\bbp$ has both $(2\bbZ)$-u.c. and $(2\bbZ+1)$-u.c..
\end{lemma}

We also present some counter-examples.

\begin{example}\label{ex:counterexamples}
$(i)$ 
Let $\alpha>1$ and $\bP(z)\ceqq C(1+|z|)^{-\alpha}\ind_{z\in\bbZ}$, where $C=C(\alpha)$ is a normalizing constant. Then $\bP$ \emph{does not} have $\bbZ$-u.c..

$(ii)$ Let $\gamma\in[0,1]$, and $\bP\ceqq \frac{1-\gamma}2(\gd_{-1}+\gd_{+1})+\gamma\gd_0$. Then $\bP$ has $\bbZ$-u.c. \emph{if and only if} $\gamma\geq1/3$.
\end{example}
The proofs of Lemma~\ref{lem:examples} and Example~\ref{ex:counterexamples} are quite straightforward and are left as an exercise to the reader (nonetheless let us mention that, with the exception of Lemma~\ref{lem:examples}.$(ii)$, all those statements are consequences of Proposition~\ref{prop:rw:fkgcharacterization} below). 

We may now state our central result. Recall that for any probability measure $\bP$ on $(E,\cE)$ and $A\in\cE$ such that $\bP(A)>0$, the measure $\bP$ conditioned to $A$ is defined with $\bP(B|A)\ceqq \bP(B\cap A)/\bP(A)$, $B\in\cE$. We have the following.

\begin{theorem}\label{thm:fkg:markovcond}
Let $X^n=(X_k)_{k=1}^n$ be the first $n$ steps of a Markov chain on some set $\cX\subset\R$ started from some $x_0\in\cX$. Assume $\cX$ is countable, locally finite, and let $\bbp$ denote the transition kernel of $X^n$. Let $A\subset\cX^{n}$, and assume the following:

(H1) for all $k\in\{0,\ldots,n-1\}$, $\bbp$ has $\Supp(\bP_{X_{k}})$-u.c.,

(H2) $\bP_{X^n}(A)>0$ and $A$ is max/min-stable,

\noindent then, $\bP_{X^n}(\cdot|A)$ is positively associated in $\R^n$.
\end{theorem}

\begin{remark}\label{rem:nocond}
One can clearly take $A\ceq\cX^n$ in the theorem, hence $\bP_{X^n}$ is positively associated as soon as $\bbp$ satisfies (H1).
\end{remark}

\begin{example}\label{ex:h1}
$(i)$ Let $X$ be a random walk on $\bbZ$ with increment distribution $\bP$. If $\bP$ has $\bbZ$-u.c., then (H1) holds. 

$(ii)$ Let $X$ be a Markov chain on $\bbZ$ such that $p_{i,i+1}+p_{i,i-1}=1$ for all $i\in\bbZ$: then (H1) holds (recall Lemma~\ref{lem:examples}.$(ii)$). 
\end{example}
\begin{remark}\label{rem:H1}
One could be tempted to replace the assumption (H1) with ``$\bbp$ has $\cX$-u.c.'', however the latter is a strictly stronger assumption since it excludes some periodic Markov chains (e.g. the simple random walk: compare Lemma~\ref{lem:examples}.$(ii)$ with Example~\ref{ex:counterexamples}.$(ii)$ for $\gamma=0$).
\end{remark}
We then deduce from Theorem~\ref{thm:fkg:markovcond} that several 1-dimensional stochastic processes are positively associated.

\begin{corollary}\label{corol:fkgcond:brown}
Let $E=\cC([0,T],\R)$ for some $T>0$. 
The following Brownian processes $(B_t)_{t\in[0,T]}$, started from $B_0=0$, are positively associated in $\cC([0,T],\R)$:

$(1)$ the Brownian bridge, $B_T=x$, $x\in\R$,

$(2)$ the Brownian excursion, $B_T=0$, $B_t\geq0$ for all $t\in[0,T]$,

$(3)$ the Brownian meander, $B_t\geq0$ for all $t\in[0,T]$,

$(4)$ the Brownian motion conditioned to remain in some (time-dependent) interval, $B_t\in[a(t),b(t)]$ for all $t\in[0,T]$, where $a,b\in\cD([0,T],[-\infty,+\infty])$ satisfy $\bP_\go(\forall\, t\in[0,T],\,\go_t\in[a(t),b(t)])>0$, for $\bP_\go$ the canonical Wiener measure on $\cC([0,T],\R)$.
\end{corollary}

\begin{corollary}\label{corol:fkgcond:bessel}
Let $E=\cC([0,T],\R_+)$. The Bessel process $(X_t)_{t\in[0,T]}$ of index $\nu\in(-1,+\infty)$, that is with generator $\frac12\frac{\dd^2}{\dd x^2} + \frac{2\nu+1}{2x}\frac{\dd}{\dd x}$, and started from some $x_0\geq 0$, is positively associated.
\end{corollary}

The case of Bessel processes with indices $\nu\leq -1$ is discussed in Remark~\ref{rem:bessel} below. Let us clarify that these corollaries of Theorem~\ref{thm:fkg:markovcond} are not meant to be exhaustive, but illustrative; the reader may think of many similar applications other than Corollaries~\ref{corol:fkgcond:brown},~\ref{corol:fkgcond:bessel}.

\subsection{The FKG lattice condition}\label{sec:compl}
Let us comment on Theorem~\ref{thm:fkg:markovcond}. 
The matter of proving association (or other correlation inequalities) for conditional distributions has been investigated in various settings, see e.g.~\cite{vdBHK06, vdBK01, Lig05} or even~\cite[Lemma~2.3]{MM18}. However, many of those results make strong assumptions on the conditioning set $A\subset E$. In Theorem~\ref{thm:fkg:markovcond}, the assumption (H2) on the set $A$ is minimal, whereas the crucial assumption (H1) only involves the transition kernel of the Markov chain \emph{without conditioning}. 

The assumption (H1) is closely related to the following correlation property, called the \emph{``FKG (lattice) condition''}.
\begin{definition}\label{def:fkgcondition}
Let $(E,\cE,\preceq, \bP)$ be a partially ordered probability space, and assume $\bP$ is a purely atomic measure.\footnote{The FKG condition can also be formulated for some non-atomic measures, but for the sake of conciseness we do not consider this case in the present paper.} We say that $\bP$ satisfies the \emph{FKG (lattice) condition} if,
\begin{equation}\label{eq:fkgcondition}
\bP(u\vee v)\bP(u\wedge v) \,\geq\, \bP(u)\bP(v)\,, \qquad\text{for }\bP\text{-a.e.\ }u,v\in E\,.
\end{equation}
\end{definition}
The FKG condition is central to the proof of Theorem~\ref{thm:fkg:markovcond}. Indeed, we prove in Lemma~\ref{lem:fkgcondition:H1} below that, in the case of a Markov chain trajectory $X^n$ in $\cX\subset\R$, (H1) is a sufficient condition for $\bP_{X^n}$ to satisfy~\eqref{eq:fkgcondition}; and, in general, the FKG lattice condition implies the FKG inequality (see notably~\cite{FKG71, Pres74}, and Lemma~\ref{lem:conclusion:inequality} below). However, those two properties are not equivalent: counter-examples in various settings have been established in the literature, see e.g.~\cite{FDS88, Lig05}. 

Determining necessary and sufficient conditions for a class of probability measures to satisfy the FKG condition is a difficult problem in general, see~\cite{JDSV82}. In the case of random walks, it is well known that any 1-dimensional random walk satisfies the FKG inequality (see Proposition~\ref{prop:fkg:rw} below); and in this paper we give a full characterization of 1-dimensional lattice random walks which satisfy the FKG condition.

\begin{proposition}\label{prop:rw:fkgcharacterization}
Let $\bP\eqqcolon (p_k)_{k\in\bbZ}$ be a (non-trivial) probability measure on $\bbZ$. Let $a\ceqq \mathrm{gcd}(w_1-w_2;w_1,w_2\in\Supp(\bP))$: in particular, there exists $b\in\bbZ$ such that $\Supp(\bP)\subset a\bbZ+b$. Let $X=(X_k)_{k\geq1}$ be a random walk with increment distribution $\bP$ started from some $x_0\in\bbZ$. Then, the following three statements are equivalent.

$(i)$ for all $n\in\N$, $X^n\coloneqq(X_k)_{k=1}^n$ satisfies the FKG lattice condition~\eqref{eq:fkgcondition},

$(ii)$ for all $b'\in\bbZ$, $\bP$ has $(a\bbZ+b')$-u.c.,

$(iii)$ the function $k\mapsto p_k$ is log-concave on $a\bbZ+b$.
\end{proposition}

Notice that Lemma~\ref{lem:examples}.$(i)$ and Example~\ref{ex:counterexamples}, as well as Lemma~\ref{lem:examples}.$(ii)$ in the random walk case, are implied by this proposition. The fact that the statement $(iii)$ implies the FKG lattice condition is quite standard, so the most interesting result is the converse implication. In particular this proposition complements Theorem~\ref{thm:fkg:markovcond} by showing that, in the case of random walks, the assumption (H1) and the FKG condition are equivalent; whereas in the case of a general Markov chains $\bP_{X^n}$, we only prove that (H1) implies~\eqref{eq:fkgcondition}. It would be interesting to determine if the converse also holds for generic Markov chains, or to find a counter-example.

As a final comment, let us mention that one can see the association of some conditional distributions as an intermediary property between the FKG condition and FKG inequality ---see e.g.~\cite{Lig05}, and the notions of ``downward FKG'' and ``downward conditional association'' therein. Therefore, since we prove in Lemma~\ref{lem:fkgcondition:H1} below that the assumption (H1) actually implies the FKG condition, it is not surprising that it yields conditional association for all set $A$ satisfying (H2). On the other hand, when (H1) fails but $\bP_{X^n}$ still satisfies the FKG inequality, it would be interesting to determine if there exists (non-trivial) sets $A\subset\cX^n$ such that $\bP_{X^n}(\cdot|A)$ is still positively associated.

\subsection*{Organization of the paper} In Section~\ref{sec:weakconv} we prove that general random walks are positively associated, which yields Theorem~\ref{thm:fkg:levy} with a weak convergence argument combined with the notion of ``infinite association''. We also prove Proposition~\ref{prop:initialvalue}. In Section~\ref{sec:markov} we prove Theorem~\ref{thm:fkg:markovcond}, and we apply it to some well-chosen conditioned Markov chains in order to deduce Corollaries~\ref{corol:fkgcond:brown} and~\ref{corol:fkgcond:bessel}. Finally, in Section~\ref{sec:rw} we prove Proposition~\ref{prop:rw:fkgcharacterization}, and in Appendix~\ref{app:convergence} we prove an integral approximation result for the FKG inequality used in Lemma~\ref{lem:fkg:lipschitz}.

\section{Association for L\'evy processes}\label{sec:weakconv}
In this section, we provide classical results on the FKG inequality, and we use them to prove Theorem~\ref{thm:fkg:levy}. We also prove Proposition~\ref{prop:initialvalue} at the end of the section.

\subsection{Preliminary results on association} 
Recall the definition of pogroups from Definition~\ref{def:pogroup}. Moreover, we say that a measure $\mu$ is \emph{regular} if it is inner regular for compact sets. We have the following.
\begin{lemma}\label{lem:fkg:convergence}
Let $(E,\Bor(E),\preceq)$ be a metric pogroup endowed with its Borel $\sigma$-algebra.
Let $(\bP_n)$, $n\geq1$ be a sequence of probability measures on $(E,\Bor(E))$, which converges to some regular probability measure $\bP_\infty$. Assume that $\bP_n$ is positively associated for all $n\geq1$; then $\bP_{\infty}$ is positively associated.
\end{lemma}
This result is a consequence of the following statement.
\begin{lemma}\label{lem:fkg:lipschitz}
Let $(E,\Bor(E),\preceq, \bP)$ be a metric pogroup endowed with its Borel $\sigma$-algebra and a regular probability measure $\bP$. The inequality~\eqref{eq:fkg} holds for all measurable, increasing functions $f,g\in L^2(\bP)$ if and only if it holds for all Lipschitz-continuous, non-negative, bounded and increasing functions.
\end{lemma}

Lemma~\ref{lem:fkg:lipschitz} follows from some standard integral approximations and measure theory arguments. In order to keep this section short, we postpone its proof to Appendix~\ref{app:convergence}.

\begin{proof}[Proof of Lemma~\ref{lem:fkg:convergence}]
Let $f,g$ be bounded, increasing and continuous functions: then, under the assumptions of Lemma~\ref{lem:fkg:convergence}, $f$, $g$ and $\bP_{n}$ satisfy the inequality~\eqref{eq:fkg} for all $n\geq1$. Taking the limit $n\to+\infty$, so do $f$, $g$ and $\bP_{\infty}$. Applying Lemma~\ref{lem:fkg:lipschitz} with the regular measure $\bP_{\infty}$, this finishes the proof.
\end{proof}

\begin{remark}
The spaces $\R^d$, $\cC([0,T],\R^d)$ and $\cD([0,T],\R^d)$ all satisfy the assumptions of Lemma~\ref{lem:fkg:convergence}: indeed, they are metric, partially ordered topological groups and are complete and separable, which implies that any endowed finite measure is regular (see e.g.~\cite{Bil99}).
\end{remark}

We also recall the following properties.

\begin{lemma}\label{lem:imageincr}
Let $(E_1,\cE_1,\preceq_1, \bP_1)$ be a partially ordered probability space. Let $(E_2,\cE_2,\preceq_2)$ be a partially ordered measurable space. Let $f:E_1\to E_2$ be a measurable, increasing function, and let $\bP_2$ be the push-forward measure of $\bP_1$ by $f$. If $\bP_1$ is positively associated, then so is $\bP_2$.
\end{lemma}
The proof of this lemma is straightforward, see e.g.~\cite[Proposition 2]{Bar05} for the details. 

\begin{lemma}\label{lem:1DFKG}
Let $(E,\cE,\preceq)$ be a \emph{totally} ordered measurable space, and let $\bP$ be \emph{any} probability measure on $E$. Then $\bP$ is positively associated.
\end{lemma}
\begin{proof} 
When the space $E$ is totally ordered, a direct computation yields for any increasing functions $f,g$ and probability measure $\bP$ that $\bE[(f(X)-f(X'))(g(X)-g(X'))]\geq0$, where $X'$ is an independent copy of $X\sim\bP$ (we leave the details to the reader), and this directly implies~\eqref{eq:fkg}.
\end{proof}

\begin{lemma}\label{lem:productFKG}
Let $(E,\cE,\preceq)$ be a product of partially ordered measurable spaces, $E=\Pi_{i=1}^n E_i$, $\cE=\bigotimes_{i=1}^n \cE_i$, endowed with the product partial order. Let $\bP=\bigotimes_{i=1}^n \bP_i$ be a product probability measure on $(E,\cE,\preceq)$. Then $\bP$ is positively associated if and only if for all $1\leq i\leq n$, $\bP_i$ is positively associated.
\end{lemma}
This result is proven in~\cite[Theorem~3]{Bar05}. Let us mention that in the case $E_i$ totally ordered for all $1\leq i\leq n$, it can also be deduced from~\cite[Theorem~3]{Pres74}.

\subsection{Association for random walks and Lévy processes}
We now turn to Theorem~\ref{thm:fkg:levy}. We prove the following.

\begin{lemma}\label{lem:fkg:levy}
Let $X=(X_t)_{t\in[0,T]}$ be a $d$-dimensional L\'evy process, $d\geq1$, started from $X_0=0$. It is positively associated in $\cD([0,T],\R^d)$ if and only if, for all $t\in[0,T]$, $X_t$ is positively associated in $\R^d$. In particular, if $d=1$ then $X$ is positively associated.
\end{lemma}
When $X_t$ is positively associated in $\R^d$ for all $t\in[0,T]$, we say that the infinitely divisible vector $X_T$ is ``infinitely associated'' (see~\cite{HPAS98}, and also~\cite{Sam95}). In particular, it is proven in~\cite[Proposition 3]{HPAS98} that the infinite association of $X_T$ is equivalent to the conditions $(i)$ and $(ii)$ from Theorem~\ref{thm:fkg:levy}. We also refer to~\cite{Pit82} in the case of Gaussian vectors (i.e. $\nu(\R^d)=0$), and~\cite{Res88, Sam95} for pure jump processes (i.e. $\gS=0$). Therefore, we only have to show Lemma~\ref{lem:fkg:levy}, and the theorem follows. 

\begin{remark}
Let us mention that, for an infinite divisible vector $Y$ in $\R^d$, $d\geq2$, with L\'evy measure $\nu$,~\cite[Theorem~2.1]{Sam95} proves that the assumption $\Supp(\nu) \subset (\R_+)^d\cup(\R_-)^d$ is \emph{not necessary} for $Y$ to be positively associated. In particular, this implies that there exists a L\'evy process $(X_t)_{t\in[0,T]}$ such that $X_T$ is positively associated in $\R^d$, but $X$ is not positively associated in $\cD([0,T],\R^d)$.
\end{remark}

Lemma~\ref{lem:fkg:levy} is a consequence of the following claim. Let us mention that in the case $d=1$ this result dates back to~\cite{Rob54}, but we present a full proof for the sake of completeness.
\begin{proposition}\label{prop:fkg:rw}
Let $\bP$ be \emph{any} probability measure on $\R^d$, and let $X^n=(X_k)_{k=1}^n$ be a random walk in $\R^d$ with independent increments with law $\bP$, started from $X_0=0$ a.s.. Then, $\bP_{X^n}$ is positively associated in $(\R^d)^n$ if and only if $\bP$ is positively associated in $\R^d$. In particular if $d=1$, then $\bP_{X^n}$ is positively associated.
\end{proposition}

\begin{proof}
Let $d\geq1$. Since the application $(x_1,\ldots,x_n)\mapsto x_1$ is increasing on $(\R^d)^n$ for the product order $\preceq$, it is clear by Lemma~\ref{lem:imageincr} that, if $\bP_{X^n}$ is positively associated, then so is $\bP_{X_1}=\bP$. Conversely, assume now that $\bP$ is positively associated and define,
\begin{equation}\label{eq:defphi}
\phi\;:\;\;\begin{aligned} &(\R^{d})^{n} \mapsto (\R^{d})^{n}\\ &(z_i)_{1\leq i\leq n} \to \textstyle \left(\sum_{i=1}^kz_i\right)_{1\leq k\leq n} \,. \end{aligned}
\end{equation}
Notice that $\phi$ is increasing, i.e.\ for $u,v\in(\R^{d})^{n}$, $u\preceq v$ implies $\phi(u)\preceq\phi(v)$. 
Let $U^n\ceqq\phi^{-1}(X^n)$; then $\bP_{U^n}=\bP^{\otimes n}$ is a product measure. By Lemma~\ref{lem:productFKG}, we deduce that $\bP_{U^n}$ is positively associated. Since $\bP_{X^n}$ is the push-forward of $\bP_{U^n}$ by $\phi$, by Lemma~\ref{lem:imageincr} it is also positively associated.

Finally, when $d=1$ we have by Lemma~\ref{lem:1DFKG} that any probability measure $\bP$ on $\R$ is positively associated, thus so is $\bP_{X^n}$.
\end{proof}

\begin{proof}[Proof of Lemma~\ref{lem:fkg:levy}] 
Let $t\in[0,T]$: we clearly have that the application $(x_s)_{s\in[0,T]} \mapsto x_t$ from $\cD([0,T],\R^d)$ to $\R^d$ is increasing for the product order. So, by Lemma~\ref{lem:imageincr}, if a random process $(X_t)_{t\in[0,T]}$ is positively associated then so is $X_t$ for any $t\in[0,T]$. 

Let us now prove the converse implication. Let $X=(X_t)_{t\in[0,T]}$ be a $d$-dimensional L\'evy process. For all $n\in\N$ and $0\leq k\leq Tn$, define, $Y^n_{k}\ceq X_{k/n}$. Then for all $n\in\N$, $Y^n$ is a random walk which satisfies $Y^n_n=X_1$. Writing $\hat X^n_t=Y^n_{\lfloor t\cdot n\rfloor}$, $t\in[0,T]$, this defines a rcll interpolation of $Y^n$, and we deduce from the classical result~\cite[Theorem~23.14]{Kal21book} that $\hat X^n$ converges weakly to $X$ for the Skorokhod topology in $\cD([0,T],\R^d)$. Since we assumed that $X_{1/n}$ is positively associated for all $n\in\N$, we deduce from Proposition~\ref{prop:fkg:rw} that the random walk $Y^n$ is positively associated. Moreover $\hat X^n$ is the image of $Y^n$ by some increasing function $\phi_n:(\R^d)^{\lfloor Tn\rfloor}\to \cD([0,T],\R^d)$, so for all $n\in\N$ it is positively associated by Lemma~\ref{lem:imageincr} (we leave the details to the reader). 
Recalling Lemma~\ref{lem:fkg:convergence}, this finishes the proof of the converse implication.

Finally, when $d=1$ the result follows directly from Lemma~\ref{lem:1DFKG}. This completes the proof of Lemma~\ref{lem:fkg:levy} and Theorem~\ref{thm:fkg:levy}.
\end{proof}

\begin{remark}
In the case of the $d$-dimensional Brownian motion $(B_t)_{t\in[0,T]}$, the proof of Theorem~\ref{thm:fkg:levy} can be simplified substantially: indeed, it follows from Lemmata~\ref{lem:1DFKG} and~\ref{lem:productFKG} that, for all $t\in[0,T]$, $B_t$ is positively associated in $\R^d$; and the proof of Lemma~\ref{lem:fkg:levy} follows from Donsker's theorem instead of~\cite[Theorem~23.14]{Kal21book}. 
\end{remark}

\subsection{Proof of Proposition~\ref{prop:initialvalue}}
This is straightforward: let $f,g:E\to\R$ be increasing functions (one can also assume  w.l.o.g.\ that they are bounded, see e.g.\ the proof of  Lemma~\ref{lem:fkg:lipschitz}), then one has
\begin{align}\label{eq:genericX_0}
\notag&\bE[f(X+Y)g(X+Y)]\,=\,\bE\left[\bE[f(X+Y)g(X+Y)|X]\right]
\\\notag&\qquad\geq\, \bE\big[\bE[f(X+Y)|X]\times\bE[g(X+Y)|X]\big]
\\&\qquad\geq\, \bE\left[\bE[f(X+Y)|X]\right]\times\bE\left[\bE[g(X+Y)|X]\right] = \bE[f(X+Y)]\bE[g(X+Y)]\,,
\end{align}
where the first inequality comes from the fact that, for almost every $X$, the functions $y\mapsto f(y+X)$, $y\mapsto g(y+X)$ are increasing on $E$ and $\bP_Y(\cdot|X)$ is positively associated for almost every $X$; and the second inequality follows from the assumption that $x\mapsto \bE[f(X+Y)|X=x]$, $x\mapsto \bE[g(X+Y)|X=x]$ are increasing $\bP_X$-almost everywhere, and that $\bP_{X}$ is positively associated. This finishes the proof of the proposition. 
\qed

\section{Conditional association for Markov chains}\label{sec:markov}
In this section we prove Theorem~\ref{thm:fkg:markovcond}, and then Corollaries~\ref{corol:fkgcond:brown},~\ref{corol:fkgcond:bessel}. Recall that, for the sake of simplicity, we only consider Markov chains and processes started from a deterministic initial value a.s..

\subsection{Proof of Theorem~\ref{thm:fkg:markovcond}} 
The proof of our main result is achieved in three short steps, which are contained in the following lemmata. Recall Definitions~\ref{def:crossings} and~\ref{def:fkgcondition}, as well as the assumptions (H1) and (H2) from Theorem~\ref{thm:fkg:markovcond}.

\begin{lemma}\label{lem:fkgcondition:H1}
Let $n\geq1$, and let $X^n=(X_1,\ldots,X_n)$ be the first $n$ steps of a Markov chain on some countable subset $\cX\subset\R$, started from $X_0=x_0\in\cX$ and with transition kernel $\bbp$. If $\bbp$ satisfies (H1), then $\bP_{X^n}$ satisfies the FKG condition.

\end{lemma}
\begin{proof}
For $(u_j)_{1\leq j\leq n}$, write $u^n\ceq (u_1,\ldots,u_n)$. Since $\bbp=(p_{x,y})_{x,y\in\cX}$ satisfies~(H1), this implies that for $\bP_{X^n}$-a.e. $u^n,v^n\in\cX^n$, one has
\begin{align}\label{eq:lem:fkgcondition:H1:1}
\notag \bP_{X^n}(u^n\vee v^n)\bP_{X^n}(u^n\wedge v^n)\,&=\, \prod_{j=0}^{n-1}p_{u_j\vee v_j, u_{j+1}\vee v_{j+1}}\,\times\, \prod_{j=0}^{n-1}p_{u_j\wedge v_j, u_{j+1}\wedge v_{j+1}}\\
&\geq\, \prod_{j=0}^{n-1}p_{u_j, u_{j+1}}\,\times\, \prod_{j=0}^{n-1}p_{v_j, v_{j+1}}\,=\,\bP_{X^n}(u^n)\bP_{X^n}(v^n) \,,
\end{align}
where $u_0=v_0\ceq x_0$, which finishes the proof.
\end{proof}

\begin{lemma}\label{lem:fkgcondition:cond}
Let $(E,\cE,\preceq, \bP)$ be a partially ordered probability space, and assume that $\bP$ is purely atomic and satisfies the FKG condition. Let $A\in\cE$ which satisfies (H2). Then, $\bP(\cdot|A)$ satisfies the FKG condition.
\end{lemma}
\begin{proof} This is straightforward. For $\bP$-a.e.\ $u,v\in A$, (H2) implies that $u\wedge v$, $u\vee v\in A$; in particular,
\begin{align*}
\bP(u\vee v|A)\bP(u\wedge v|A) = \frac1{\bP(A)^2} \bP(u\vee v)\bP(u\wedge v)\geq \frac1{\bP(A)^2} \bP(u)\bP(v) = \bP(u|A)\bP(v|A)\,,
\end{align*}
which finishes the proof.
\end{proof}

\begin{lemma}\label{lem:conclusion:inequality}
Let $\bP$ be a probability measure on a countable, locally finite subset of $(\R^n,\Bor(\R^n),\preceq)$ which satisfies the FKG condition. Then it is positively associated.
\end{lemma}

\begin{proof}
This is essentially the content of~\cite{Pres74}. Let $g:\R^n\to\R$ be a non-negative, bounded and increasing function, and define the probability measure $\bQ$ on $\R^n$ with
\[
\frac{\dd \bQ}{\dd \bP}(u)\,=\, \frac{g(u)}{\bE g}\,,\qquad\text{for }\bP\text{-a.e.\ }u\in \R^n\,,
\]
where $\bE=\bE_\bP$ denotes the expectation with respect to $\bP$ (we assume that $\bE g>0$, otherwise $g\equiv0$ $\bP$-a.e.\ and~\eqref{eq:fkg} holds for all $f$). Since $\bP$ satisfies~\eqref{eq:fkgcondition} and $g$ is increasing, one clearly has for $\bP$-a.e.\ $u,v\in\R^n$,
\begin{align*}
\bQ(u\vee v)\bP(u\wedge v) = \frac{g(u\vee v)}{\bE g}\bP(u\vee v)\bP(u\wedge v)\geq \frac{g(u)}{\bE g}\bP(u)\bP(v) = \bQ(u)\bP(v)\,.
\end{align*}
Therefore, it follows\footnote{Indeed, since $\mathrm{Supp}(\bP)$ is locally finite, there exists a $\gs$-finite counting measure $\mu$ on some set $A\subset\R$ such that $\mathrm{Supp}(\bP)\subset A^n$. Letting $\go\ceq\mu^{\otimes n}$ in~\cite{Pres74}, this yields the result (details are left to the reader).} from~\cite[Theorem~3]{Pres74} that for $f:\R^n\to\R$ a bounded increasing function, one has $\bE_\bQ f\geq\bE_\bP f$, which is exactly~\eqref{eq:fkg}. Recalling Lemma~\ref{lem:fkg:lipschitz}, the inequality~\eqref{eq:fkg} also holds for all increasing $f,g\in L^2(\bP)$, finishing the proof.
\end{proof}

With these lemmata at hand, the proof of Theorem~\ref{thm:fkg:markovcond} is immediate.

\begin{proof}[Proof of Theorem~\ref{thm:fkg:markovcond}]
Under the assumptions (H1, H2), Lemmata~\ref{lem:fkgcondition:H1},~\ref{lem:fkgcondition:cond} imply that the measure $\bP_{X^n}(\cdot|A)$ satisfies the FKG condition. By Lemma~\ref{lem:conclusion:inequality}, it is positively associated, completing the proof.
\end{proof}

\subsection{Proof of the corollaries}
Let $T>0$, let $X^n=(X_1,\ldots,X_n)$ be a random vector in $\R^n$, and $X_0\ceq x_0\in\R$. One can define a rcll (resp. continuous) interpolation of $X$ in $\cD([0,T],\R)$ (resp. $\cC([0,T],\R)$) by letting, for $t\in[0,T]$,
\begin{align}\label{eq:interpolations}
\hat X_t^n\,&\ceq\, \sum_{j=0}^n X_j \ind_{[Tj/n, T(j+1)/n)}(t)\,, \\
\notag\text{resp.}\qquad \tilde X_t^n\,&\ceq\, \sum_{j=0}^n \left[X_j + (\tfrac{nt}{T}-j)(X_{j+1}-X_j) \right] \ind_{[Tj/n, T(j+1)/n)}(t)\,.
\end{align}
Both are increasing functions of $X^n$; hence by Lemma~\ref{lem:imageincr}, if the vector $X^n$ is positively associated, then the processes $\hat X^n$ and $\tilde X^n$ are positively associated as well. Moreover by Lemma~\ref{lem:fkg:convergence}, if the sequence of processes $\hat X^n$ (or $\tilde X^n$) converges weakly, then the limiting process is positively associated as well. Therefore, each corollary of Theorem~\ref{thm:fkg:markovcond} will be proven by determining a sequence of Markov chains $X^n$ with transition kernel $\bbp^n$, and sets $A_n$ in $\cE_\cD$ (or $\cE_\cC$) such that,

--- for all $n\geq1$, $\bbp^n$ satisfies (H1),

--- for all $n\geq1$, $A_n$ satisfies (H2) under $\bP_{\hat X^n}$ (or $\bP_{\tilde X^n}$),

--- the sequence of processes $\hat X^n$ (or $\tilde X^n$) conditioned on $A_n$, $n\geq1$, converges weakly to the desired process.

\begin{proof}[Proof of Corollary~\ref{corol:fkgcond:brown}]
Let us prove the corollary simultaneously in the four cases, which we denote respectively with the subscripts $j\in\{1,2,3,4\}$. Let $Y^n$ be the first $n$ steps of a random walk with discrete Laplace increment for some $\gb>0$.\footnote{Of course, one could also consider a simple random walk (recall Lemma~\ref{lem:examples}.$(ii)$), provided that they restrict themself to $n\in2\N$ in the cases $j\in\{1,2\}$ below.} We may take $\gb>0$ such that the Laplace increments have unit variance, and we set $X^n\ceq n^{-1/2}Y^n$. Then for all $n\geq1$, it follows from Lemma~\ref{lem:examples}.$(i)$ that the transition kernel $\bbp^n$ of the Markov chain $X^n$ satisfies (H1).

We then define the sets $A_{j,n}=A_j$, $n\geq1$ differently for each case $j\in\{1,2,3,4\}$. Let,
\begin{equation}
A_j\,\ceqq\,\begin{cases}
\{x\in\cD([0,T],\R)\,|\,x_T=0\} & \text{for }j=1,\\
\{x\in\cD([0,T],\R)\,|\,x_T=0; \forall\,t,\, x_t\geq0\} & \text{for }j=2,\\
\{x\in\cD([0,T],\R)\,|\,\forall\,t,\, x_t\geq0\} & \text{for }j=3,\\
\{x\in\cD([0,T],\R)\,|\,\forall\,t,\, x_t\in[a(t),b(t)]\} & \text{for }j=4,
\end{cases}
\end{equation}
Then, for $n$ large enough, one clearly has in the four cases that $\bP_{\hat X^n}(A_j)>0$ and $A_j$ is max/min-stable, so (H2) holds.

By Theorem~\ref{thm:fkg:markovcond}, we deduce in each case $j\in\{1,2,3,4\}$ that $\bP_{\hat X^n}(\cdot|A_j)$ (similarly $\bP_{\tilde X^n}(\cdot|A_j)$) is positively associated for all $n\in\N$. Therefore, it only remains to confirm in each case that the conditioned random walk converges weakly to the desired limit. Those results are classical: for the Brownian bridge $j=1$, this is the content of~\cite{Lig68}. For the Brownian excursion $j=2$, this can be found in~\cite{Kai76}. For the Brownian meander $j=3$, we refer to~\cite{Igl74}. Finally in the case $j=4$, by assumption the event $A_4$ has positive probability for the Wiener measure, therefore the weak convergence can be deduced directly from Donsker's theorem (details are left to the reader). This concludes the proof of the corollary.
\end{proof}

\begin{proof}[Proof of Corollary~\ref{corol:fkgcond:bessel}]
For $\nu\in\R$, we may define a Markov chain $(Y_k)_{k\geq0}$ on $\N$ with transition probabilities that satisfy,
\begin{align*}
p_{i,i+1}\,&=\,1-p_{i,i-1}\,\ceq\, \frac12\left(1+\frac{2\nu+1}{2i}+o\big(i^{-1}\big)\right)\,,\qquad\text{if }i\geq1,\\
p_{0,1}\,&=\,1\,,
\end{align*}
and $p_{i,j}=0$ otherwise. Since $Y$ only makes unitary steps, by Lemma~\ref{lem:examples}.$(ii)$ its transition kernel satisfies (H1). For $n\geq1$, letting $X^n\ceq n^{-1/2}(Y_1,\ldots,Y_n)$ and $\tilde X^n$ its rcll interpolation in $\cD([0,T],\R)$, one deduces from Theorem~\ref{thm:fkg:markovcond}, Remark~\ref{rem:nocond} and Lemma~\ref{lem:imageincr} that $\bP_{\tilde X^n}$ is positively associated (this is a carbon copy of previous arguments, details are left to the reader). Moreover when $\nu>-1$, the convergence of the process $\tilde X^n$ to the Bessel process of index $\nu$ is proven in~\cite[Theorem~5.1]{Lam62}. By Lemma~\ref{lem:fkg:convergence}, this concludes the proof of the corollary.
\end{proof}

\begin{remark}\label{rem:bessel}
Corollary~\ref{corol:fkgcond:bessel} is also expected to hold for indices $\nu\leq -1$ (i.e. when the barrier at 0 is absorbing), however this requires a convergence result to Bessel processes different from~\cite{Lam62}. Since this is not the object of the present paper, we leave this consideration to later work.
\end{remark}

\section{Characterization of the FKG condition for lattice random walks}
\label{sec:rw}
In this section we prove Proposition~\ref{prop:rw:fkgcharacterization}. Recall Definitions~\ref{def:crossings} and~\ref{def:fkgcondition}.

\begin{proof}[Proof of Proposition~\ref{prop:rw:fkgcharacterization}]
First, recall the assumption (H1) from Theorem~\ref{thm:fkg:markovcond}, and notice that it is implied by $(ii)$; so we deduce from Lemma~\ref{lem:fkgcondition:H1} that $(ii)\Rightarrow(i)$. 

We now prove $(i)\Rightarrow(ii)$. Recall that $a\ceqq \mathrm{gcd}(w_1-w_2;w_1,w_2\in\Supp(\bP))$, and the notation $u^k\ceq (u_1,\ldots,u_k)$. If $(ii)$ fails, then there exists some large $k_0\in\N$ and $u_{k_0}\leq v_{k_0}$ in $\mathrm{Supp}(X_{k_0})$, $u_{{k_0}+1}\geq v_{{k_0}+1}$ in $\bbZ$, such that,
\[
\bP(u_{{k_0}+1}-u_{k_0})\bP(v_{{k_0}+1}-v_{k_0}) \,>\, \bP(u_{{k_0}+1} - v_{k_0})\bP(v_{{k_0}+1}-u_{k_0})\,\geq\,0\,.
\]
By definition of $\Supp(X_{k_0})$ there exists two paths $u^{k_0}=(u_1,\ldots,u_{k_0})$ and $v^{k_0}=(v_1,\ldots,v_{k_0})$ such that $\bP_{X^{k_0}}(u^{k_0})>0$, $\bP_{X^{k_0}}(v^{k_0})>0$. For $j\geq {k_0}$, let us define by induction,
\[\begin{aligned}
&u_{j+1}\ceq\begin{cases}
u_j+(u_{{k_0}+1}-u_{k_0}) &\text{if }j-{k_0}\text{ is even,}\\
u_j+(v_{{k_0}+1}-v_{k_0}) &\text{if }j-{k_0}\text{ is odd,}
\end{cases}
\\\text{and}\quad &v_{j+1}\ceq
\begin{cases}
v_j+(v_{{k_0}+1}-v_{k_0}) &\text{if }j-{k_0}\text{ is even,}\\
v_j+(u_{{k_0}+1}-u_{k_0}) &\text{if }j-{k_0}\text{ is odd.}
\end{cases}
\end{aligned}\]
Then, one has $u_j\leq v_j$ if $j-{k_0}$ is even, and $v_j\leq u_j$ if $j-{k_0}$ is odd. Therefore, one observes for $m\geq k_0$ that,
\[
\frac{\bP_{X^{m}}(u^{m}\wedge v^{m})\bP_{X^{m}}(u^{m}\vee v^{m})}{\bP_{X^{m}}(u^{m})\bP_{X^{m}}(v^{m})}\,\leq\, C\left(\frac{\bP(u_{{k_0}+1}- v_{k_0})\bP(v_{k_0+1}-u_{k_0})}{\bP(u_{k_0+1}-u_{k_0})\bP(v_{k_0+1}-v_{k_0})}\right)^{m}\,,
\]
for some $C>0$. Since the r.h.s. goes to $0$ as $m\to+\infty$, there exists a large $m$ and $u^m$, $v^m$ such that~\eqref{eq:fkgcondition} fails, which conclude the proof of $(i)\Rightarrow(ii)$.

Let us now show that $(iii)\Rightarrow(ii)$. Let $b'\in\bbZ$, $u_1,v_1\in a\bbZ+b'$ and $u_2,v_2\in\bbZ$. We may assume without loss of generality that $u_1>v_1$, and that $u_2<v_2$ (otherwise~\eqref{eq:def:crossings} clearly holds). Let us define,
\begin{equation}\label{eq:prop:rw:fkgcharacterization:1}
w_1=u_2-u_1\;,\;  w_2=u_2-v_1\;,\;w_3=v_2-u_1\;,\; w_4 = v_2-v_1\;.
\end{equation}
In particular, one has $w_i\in(a\bbZ+b)$ for $i\in\{1,\ldots,4\}$, $w_1<w_i<w_4$ for $i\in\{2,3\}$ and $w_1+w_4=w_2+w_3$. Moreover~\eqref{eq:def:crossings} can be rewritten $p_{w_1}p_{w_4}\leq p_{w_2}p_{w_3}$. If $p_{w_1}=0$ or $p_{w_4}=0$ then~\eqref{eq:def:crossings} clearly holds, otherwise one has
\begin{equation}\label{eq:prop:rw:fkgcharacterization:2}
p_{w_2}p_{w_3}\,=\, p_{w_1}p_{w_4} \frac{p_{w_2}p_{w_3}}{p_{w_1}p_{w_4}}\,\geq\, p_{w_1}p_{w_4} \left(\frac{p_{w_2}}{p_{w_1}}\right)^{1-\frac{w_4-w_3}{w_2-w_1}}\,=\,p_{w_1}p_{w_4}\,,
\end{equation} 
where we used that $k\mapsto p_k$ is log-concave on $a\bbZ+b$.\smallskip

It remains to prove $(ii)\Rightarrow(iii)$. In order to lighten upcoming formulae, we assume without loss of generality that $a=1$ and $b=0$. 
Let $w_1\in \bbZ$, and let us assume w.l.o.g.\ that $p_{w_1}>0$. We first prove by induction that for all $i\geq 1$,
\begin{equation}\label{eq:prop:rw:fkgcharacterization:3}
\frac{p_{w_1+1}}{p_{w_1}}\,\geq\, 
\left(\frac{p_{w_1+i}}{p_{w_1}}\right)^{i^{-1}}\,.
\end{equation}
This clearly holds for $i=1$. We now assume that~\eqref{eq:prop:rw:fkgcharacterization:3} holds for some $i\in\N$, and we define $w_2\ceq w_1+1$, $w_3\ceq w_1+i$ and $w_4\ceq w_1+(i+1)$. With this notation, we claim that one can define some points $u_1,u_2,v_1,v_2\in \bbZ$ such that the identities~\eqref{eq:prop:rw:fkgcharacterization:1} hold, and such that $(ii)$ yields $p_{w_1}p_{w_4}\leq p_{w_2}p_{w_3}$ (we leave the details to the reader). Therefore,
\[
\frac{p_{w_2}}{p_{w_1}} \,\geq\, \frac{p_{w_4}}{p_{w_3}} \,=\, \frac{p_{w_4}}{p_{w_1}} \frac{p_{w_1}}{p_{w_3}}\,\geq\, \frac{p_{w_4}}{p_{w_1}} \left(\frac{p_{w_1}}{p_{w_2}}\right)^{i}\,,
\] 
where the second inequality follows from the induction hypothesis. Rearranging the terms, this shows that~\eqref{eq:prop:rw:fkgcharacterization:3} holds for the index $(i+1)$, finishing the proof of~\eqref{eq:prop:rw:fkgcharacterization:3} by induction. 

Finally, the log-concavity of $k\mapsto p_k$ follows naturally from~\eqref{eq:prop:rw:fkgcharacterization:3}.
Indeed, let $w_1,w_2\in \bbZ$ and assume that there exists 
$k\in\{w_1+1,\ldots,w_2-1\}$ such that,
\begin{equation}\label{eq:prop:rw:fkgcharacterization:4}
\frac{\log p_{k}-\log p_{w_1}}{k-w_1}\,<\, 
\frac{\log p_{w_2}-\log p_{w_1}}{w_2-w_1}\,.
\end{equation}
We may assume w.l.o.g.\ that $k$ is the smallest index satisfying~\eqref{eq:prop:rw:fkgcharacterization:4}: in particular $(k,\log p_{k})$ is strictly below the straight line from $(w_1,\log p_{w_1})$ to $(w_2,\log p_{w_2})$, whereas $(k-1,\log p_{k-1})$ is above. This implies that
\[
\log p_{k}-\log p_{k-1} \,<\, \frac{\log p_{w_2}-\log p_{k-1}}{w_2-(k-1)}\,,
\]
which directly contradicts~\eqref{eq:prop:rw:fkgcharacterization:3}. 
This finishes the proof that $k\mapsto p_k$ is log-concave.
\end{proof}

\appendix
\section{Proof of Lemma~\ref{lem:fkg:lipschitz}}\label{app:convergence}
In this section we prove Lemma~\ref{lem:fkg:lipschitz}. It follows from the following integral approximation result.

\begin{lemma}\label{lem:fkg:functionincreasing}
Let $(E,\Bor(E),\preceq, \mu)$ be a metric pogroup endowed with its Borel $\sigma$-algebra and a finite, regular measure $\mu$. Then for any function $f$ measurable, increasing, non-negative, bounded and $\eps>0$, there exists $h$ increasing, Lipschitz-continuous, non-negative and bounded such that $\int_E|f-h| d\mu<\eps$.
\end{lemma}
Notice that this naturally implies Lemma~\ref{lem:fkg:lipschitz}. 
Indeed, if $f$, $g$ are increasing, non-negative and bounded, they can be approximated for any $\eps>0$ by some increasing Lipschitz functions $h_1, h_2$ such that $\int_E|f-h_1| d\bP<\eps$ and $\int_E|g-h_2| d\bP<\eps$. If, for any $\eps>0$, $h_1, h_2$ satisfy~\eqref{eq:fkg}, then letting $\eps\to0$ yields that $f,g$ satisfy~\eqref{eq:fkg} as well, since they are bounded and $\bP$ is finite. Afterward, the FKG inequality can straightforwardly be extended to $f,g$ increasing, bounded and signed, by applying it to $f+\|f\|_\infty$, $g+\|g\|_\infty$; then to any increasing functions $f,g\in L^2(\bP)$, by noticing that if $f$ is increasing, then so is $(f\wedge K) \vee (-K)$ for all $K>0$, and by applying the dominated convergence theorem. Let us now prove Lemma~\ref{lem:fkg:functionincreasing}.

\begin{proof}[Proof of Lemma~\ref{lem:fkg:functionincreasing}]
If the function $h$ is not required to be increasing, then this result is standard. Indeed, since $\mu$ is a finite measure on a metric space, it is outer regular: for any $B\in\Bor(R)$, there exists $U$ open such that $B\subset U$ and $\mu(U\setminus B)\le\eps/2$. Then one can approximate the function $f=\ind_B$ with $h_K=1\vee (Kd(\cdot,U^c))$, which is Lipschitz-continuous and satisfies $\int_E|\ind_B-h_K|d\mu<\eps$ for $K$ large enough by the dominated convergence theorem. Finally, one can extend the result to any measurable, non-negative, bounded function by approaching it with a positive linear combination of indicator functions.

Let us now consider $B\in\Bor(R)$ an increasing set. Recall that we may assume, without loss of generality, that the metric $d(\cdot,\cdot)$ on $E$ is translation invariant, see e.g.~\cite{Kak36}. We have the following result, which is proven afterwards.
\begin{lemma}\label{lem:fkg:openincreasing}
Let $\eps>0$ and let $(E,\Bor(E),\preceq, \mu)$ be a metric pogroup endowed with its Borel $\sigma$-algebra and a finite, regular measure $\mu$. Then for any increasing set $B\in\Bor(E)$ and $\eps>0$, there exists an increasing open set $U$ such that $B\subset U$ and $\mu(U\setminus B)\le \eps$.
\end{lemma}
Moreover, we also claim that if $U$ is increasing, then $x\mapsto d(x,U^c)$ is an increasing function. Indeed, assume that there exists $x\preceq y$ such that $d(x,U^c)>d(y,U^c)$. Then there is $z\in U^c$ with $d(y,z)<d(x,U^c)$. This implies $z+x-y\in U^c$ because $z+x-y\preceq z$ and $U$ is increasing, and $d(x,z+x-y)=d(y,z)<d(x,U^c)$, which yields a contradiction.

Therefore, if $B$ is an increasing set, letting $U\supset B$ be an increasing open set such that $\mu(U\setminus B)\le \eps/2$, then the function $h_K=1\vee (Kd(\cdot,U^c))$ is increasing and satisfies $\int_E|\ind_B-h_K|d\mu<\eps$ for $K$ large enough; and since any non-negative, increasing function can be approximated with positive linear combinations of increasing indicator functions, this concludes the proof of Lemma~\ref{lem:fkg:functionincreasing} subject to Lemma~\ref{lem:fkg:openincreasing}.
\end{proof}

\begin{proof}[Proof of Lemma~\ref{lem:fkg:openincreasing}]
Let $B\in\Bor(E)$ be an increasing set, and $\eps>0$. Since $\mu$ is regular, there exists $K\subset B^c$ compact such that $\mu(B^c\setminus K)\leq \eps$. Let us define
\[
K^\downarrow\coloneqq \{x\in E; \exists y\in K, y\succeq x\}\,.\]
One notices that $K^\downarrow$ is a decreasing set that contains $K$. Moreover, let us prove that it is closed: Let $(x_n)_{n\ge1}$ be a sequence in $K^\downarrow$, and assume that there exists $x\in E$ such that $x_n\to x$ as $n\to+\infty$. For each $n\ge1$, there exists $y_n\in K$ such that $y_n\succeq x_n$; and since $K$ is sequentially compact, we may assume (by extracting a subsequence) that $(y_n)_{n\ge1}$ converges to some $y\in K$. Then, $y_n-x_n\in\cC_+$ for all $n$, where we recall that $\cC_+$ is the non-negative cone. Since $\cC_+$ is closed, this implies that $y\succeq x$, hence $x\in K^\downarrow$.

Since $B^c$ is also a decreasing set, one notices that $K^\downarrow\subset B^c$. Therefore, 
\[
\mu\big(\big(K^\downarrow\big)^c\setminus B\big) = \mu\big(B^c\setminus K^\downarrow\big) \leq \mu(B^c\setminus K) \leq \eps\,.
\]
Since $(K^\downarrow)^c$ is increasing, open and contains $B$, this finishes the proof.
\end{proof}

\section*{Acknowledgments}
The author is thankful to Christian Houdr\'e, Christophe Garban, Diederik van Engelenburg and Romain Panis for fruitful discussions, and to Niccol\`{o} Pianalto and an anonymous referee for pointing out an oversight in a previous version.

\section*{Funding}
The author was supported by the grant ANR-22-CE40-0012 (project Local).

\end{document}